\documentclass{llncs}

\usepackage[usenames]{color}

\usepackage{latexsym}
\usepackage{amsfonts}
\usepackage{graphicx}
\usepackage{amsmath}
\usepackage{amssymb}

\newcommand{\Z}{\mathbb{Z}}

\newcommand{\R}{\mathbb{R}}

\newcommand{\bp}{\begin{problem}}

\newcommand{\ep}{\end{problem}}

\newcommand{\ba}{\begin{answer}}

\newcommand{\ea}{\end{answer}}

\newcommand{\ben}{\renewcommand{\theenumi}{\alph{enumi}}

\renewcommand{\labelenumi}{(\theenumi)}\begin{enumerate}}

\newcommand{\een}{\end{enumerate}}

\newcommand{\Mod}{\mathrm{Mod}}
\newcommand{\Prim}{\mathrm{Prim}}

\begin{document}

\title{A Cryptographic Application of the Thurston norm}
\author{Ram\'{o}n Flores\inst{1}  \and Delaram Kahrobaei\inst{2} \and Thomas Koberda\inst{3}}
\institute{Department of Geometry and Topology, University of Seville, Spain\\
\email{ramonjflores@us.es}
\thanks{Ram\'{o}n Flores is partially supported by FEDER-MINECO Grant MTM2016-76453-C2-1-P.}
\and
The University of York, UK, CUNY Graduate Center and City University of New York, New York University\\
\email{delaram.kahrobaei@york.ac.uk, dk2572@nyu.edu}
\thanks{Delaram Kahrobaei was supported by the ONR (Office of Naval Research) grant N000141512164.}
\and
Department of Mathematics, University of Virginia\\
\email{thomas.koberda@gmail.com}
\thanks{Thomas Koberda is partially supported by an Alfred P. Sloan Foundation Research Fellowship and by NSF Grant DMS-1711488.}
}

\maketitle

\begin{abstract}
We discuss some applications of 3-manifold topology to cryptography. In particular, we propose a public-key and a 
symmetric-key cryptographic scheme based on the Thurston norm on the first cohomology of hyperbolic manifolds. 
\end{abstract}

\section{Introduction}

Geometric group theory and low dimensional topology have developed many powerful tools for studying groups, and many group theoretic ideas have been productive in group-based cryptography. Here, we propose importing ideas from hyperbolic geometry to build new cryptoschemes with certain security advantages.

Specifically, we consider the Thurston norm on $H^1(M,\mathbb{R})$, the first cohomology of a finite volume hyperbolic $3$--manifold, as introduced in~\cite{Thurston}. The Thurston norm measures the Euler characteristic of the simplest surface in $M$ which represents a second homology class which is Poincar\'e dual to an integral cohomology class, and extends it to the entirety of $H^1(M,\mathbb{R})$. The Thurston norm has a remarkable linear nature which makes computations with it tractable, and as outlined below, organizes the fibrations of a fibered hyperbolic $3$--manifold.

We will use the Thurston norm to build a new symmetric-key cryptographic scheme. Combined with a certain group based public key exchange, we obtain a public-key cryptoscheme which has two levels of security and in which all communications are over public channels.

National Security Agency (NSA) announced plans to upgrade current security standards in 2015; the goal is to replace all deployed cryptographic protocols with quantum secure protocols, due to the increasing possibility of quantum attacks. This transition requires a new, post-quantum, security standard to be accepted by the National Institute of Standards and Technology (NIST). Proposals for quantum secure cryptosystems and protocols have been submitted for the standardization process. There are six main primitives currently proposed to be quantum-safe: (1) lattice-based (2) code-based (3) isogeny-based (4) multivariate-based (5) hash-based, and (6) {\it group-based} cryptographic schemes.
Applications to post-quantum group-based cryptography could be shown if the underlying security problem is NP-complete or unsolvable; ideally one could analyze the relationship of the problems under consideration here to the hidden subgroup problem (HSP), then analyze Grover's search problem. As for the relationship to the HSP, the groups under consideration are infinite and so a practical way to process them needs to be developed.

In~\cite{HKMPPQ}, a practical cryptanalysis of WalnutDSA was proposed, a platform which was given in 2016 in ~\cite{AAGG} as a post-quantum cryptosystem using braid groups and conjugacy search problem, submitted to NIST competition in 2017.

There are other group-theoretic problems and classes of groups which have been proposed for post-quantum group-based cryptography, as we summarize here. The pioneers in this field were Wagner-Magyarik, who in~\cite{WM} used a right-angled Artin group as a platform,
relying on the word problem and the word choice problem; this approach was later improved by Levy-Perret in \cite{LP}. At the same time, Birget-Magliveras-Sramka~\cite{BMS} proposed a new protocol based on different groups that share some properties with Higman-Thompson
groups.
Later on, Flores-Kahrobaei and Flores-Kahrobaei-Koberda proposed right-angled Artin groups as a 
latform for various cryptographic protocols~\cite{FK,FKK}.
Eick and Kahrobaei proposed polycyclic groups as a platform, using the conjugacy search problem as a basis for security~\cite{EK}.
Gryak-Kahrobaei proposed other group-theoretic problems for consideration for polycyclic group platorms~\cite{GK}.
Kahrobaei-Koupparis~\cite{KK} proposed a post-quantum digital signature using polycyclic groups.
Kahrobaei-Khan~\cite{KK6} proposed a public-key cryptosystem using polycyclic groups~\cite{KK6}. Habeeb-Kahrobaei-Koupparis-Shpilrain  proposed public key exchanges using semidirect products of semigroups in~\cite{HKKS}.
Thompson's groups have been considered by Shpilrain-Ushakov,
with cryptoschemes based on the decomposition search problem~\cite{SU}.
Hyperbolic groups have been proposed by Chatterji-Kahrobaei-Lu,
relying on properties of subgroup distortion and the geodesic length problem~\cite{CKL}.
Cavallo-Habeeb-Kahrobaei-Shpilrain proposed using small cancellation groups for secret sharing scheme \cite{CK,HKS}.
Free metabelian groups have been proposed as a platform by Shpilrain-Zapata, with the scheme based on the subgroup membership search problem~\cite{SZ}.
Kahrobaei-Shpilrain proposed free nilpotent $p$-groups as a platform for a semidirect product public key~\cite{KS16}. Linear groups were proposed by Baumslag-Fine-Xu~\cite{BFX}, and Grigorchuk's group have been proposed in~\cite{P}. Finally in \cite{KM}, arithmetic groups were proposed as platform for a symmetric-key cryptographic scheme.

\section{Hyperbolic $3$--manifolds and the Thurston norm}

We review some well--known background about hyperbolic $3$--manifolds and the Thurston norm~\cite{Thurston}, concentrating on the case of fibered hyperbolic $3$--manifolds. 

\subsection{Generalities about the Thurston norm}
Let $M=M_{\psi}$ be a fibered hyperbolic $3$--manifold. That is to say, there is an orientable surface $S$ with negative Euler characteristic and a mapping class $\psi\in\Mod(S)$ such that $M$ is the mapping torus of $\psi$. Observe that the rank of $H^1(M,\mathbb{R})$ is at least one, since $\pi_1(M)$ surjects to $\mathbb{Z}$. Rational cohomology classes of $M$ which correspond to fibrations of $M$ are called \emph{fibered cohomology classes} of $M$. Precisely what is meant by this correspondence is the following: first, replace the given cohomology class
by the smallest nonzero  multiple which is integral. A cohomology class $\phi\in H^1(M,\mathbb{Z})$
is a homomorphism to $\mathbb{Z}$, which by standard arguments from algebraic topology is induced by a based map of spaces $\Phi\colon
M\to S^1$. After
modifying $\Phi$ by a homotopy, the generic preimage of a point will be a subsurface $S$
 of $M$ which  is Poincar\'e dual to $\phi$. If $\Phi$ is
chosen carefully enough, $S$ will be a fiber of a fibration over $S^1$. The fibration  can also be built by pulling the form $d\theta$ from $S^1$
back under $\Phi$ and integrating it. See~\cite{Thurston} for more details.

A fibered $3$--manifold $M$ is called \emph{atoroidal} if it does not contain a non-peripheral incompresible torus. Here, this means that if
$T\subset M$ is a $\pi_1$-injective copy of the torus, then $T$ can be pushed into a cusp of $M$.
It is a famous result of Thurston
that a fibered $3$--manifold admits a finite volume hyperbolic metric if and only if it is atoroidal, which in turn will happen if and only if
no power of the
mapping class $\psi$ fixes the homotopy class of a simple closed loop on $S$. Here, a simple closed loop on $S$ is an essential
copy of $S^1$ which is not parallel to a puncture or boundary component of $S$. Such a mapping class $\psi$ is called
\emph{pseudo-Anosov}.

It is a standard result from foliation theory that if the rank of $H^1(M,\mathbb{R})$ is at least two, then small perturbations of a
fibered cohomology
class $\phi\in H^1(M,\mathbb{Q})\subset H^1(M,\mathbb{R})$
will give new fibrations of $M$ over the circle which are inequivalent to $\phi$. Thurston's work~\cite{Thurston}
organized the fibrations of $M$ by defining a norm $||\cdot||_T$ on $H^1(M,\mathbb{R})$, called the \emph{Thurston norm}. The norm of a cohomology class $||\phi||_T$ is given by $\min_{S_{\phi}} |\chi(S_{\phi})|$, where this minimum is taken over surfaces which are Poincar\'e dual to $\phi$. The following summarizes the relevant features of the Thurston norm:
\begin{theorem}
Let $M$ be a compact atoroidal $3$--manifold with $\chi(M)=0$, and let $||\cdot||_T$ be the Thurston norm.
\begin{enumerate}
\item
$||\cdot||_T$ is a nondegenerate norm on the vector space $H^1(M,\mathbb{R})$;
\item
The unit norm ball is a convex polytope, all of whose vertices lie at rational points in $H^1(M,\mathbb{R})$;
\item
Let $\phi\in H^1(M,\mathbb{Z})$ be a fibered cohomology class of $M$. Then there is a maximum dimensional face $F$ of the unit ball of $||\cdot||_T$ such that $\phi\in\mathbb{R}\cdot F$. Moreover, every primitive integral cohomology class $\phi\in\mathbb{R}\cdot F$ is fibered. The face $F$ is called a \emph{fibered face} of the unit norm ball.
\end{enumerate}
\end{theorem}
Here, an integral cohomology class if called \emph{primitive} if it is nonzero and if it is
 not an integer multiple of another integral cohomology class. Viewed as a tuple of vectors, an integral cohomology class is primitive if and only if the entries of the tuple are relatively prime and not all zero.

Let $\phi\in H^1(M,\mathbb{Z})$ be a fibered cohomology class. Then $M$ fibers over the circle with fiber $S=S_{\phi}$, where $\pi_1(S)<\pi_1(M)$ is identified with $\ker\phi$. The following proposition is standard and we include its proof for the convenience of the reader.

\begin{proposition}\label{prop:linear}
Suppose $M$ is hyperbolic, and let $S$ be the fiber of a fibration of $M$ over $S^1$. Then $\pi_1(S)<\pi_1(M)$ is exponentially distorted, and the membership problem for $\pi_1(S)$ is solvable in linear time.
\end{proposition}
\begin{proof}
We have that $\pi_1(M)$ is a semidirect product of $\mathbb{Z}$ with $\pi_1(S)$, where the conjugation action of a generator $t$ of $\mathbb{Z}$ is given by a pseudo-Anosov mapping class of $\pi_1(S)$. If $1\neq \gamma\in\pi_1(S)$ and $\psi$ is a pseudo-Anosov mapping class which has been lifted to an automorphism of $\pi_1(S)$, then the length of the shortest representative in the conjugacy class of $\psi^n(\gamma)$ grows like $\lambda_{\psi}^n$, where $\lambda_{\psi}>1$ is a real number called the \emph{stretch factor} of $\psi$. Since conjugation by $t$ acts on $\pi_1(S)$ by $\psi$, we have that $\psi^n(\gamma)=t^{-n}\gamma t^n$, a word whose length is linear in $n$. It follows that $\pi_1(S)<\pi_1(M)$ is exponentially distorted. We refer the reader to Expos\'e 10 of~\cite{FLP} for details on word growth entropy
and pseudo-Anosov mapping classes.

Now suppose that $g\in\pi_1(M)$ is a given element, and we wish to determine if $g\in\pi_1(S)$. The group $\pi_1(M)$ surjects to $\Z$ by a homomorphism $\phi$, and the kernel of this map is exactly $\pi_1(S)$. If $\pi_1(M)=\langle g_1,\ldots,g_k\rangle$, then $\phi$ is determined by its values on the generators of $\pi_1(M)$. If $g\in\pi_1(M)$ is a product of $N$ generators of $\pi_1(M)$, we compute the value of $\phi$ on the $N$ generators needed to represent $g$ and add them up, which requires computational resources bounded by a linear function in $N$. If the resulting sum is zero, then $g\in\pi_1(S)$. If the sum is nonzero then $g\notin\pi_1(S)$.
\end{proof}

\subsection{Examples}
From the general description, the Thurston norm seems very difficult to compute and hence unwieldy for many practical applications.
However, there are many situations in which the Thurston norm can be computed, at least in a cone over a fibered face.

In~\cite{McPaper}, McMullen defined the Teichm\"uller polynomial associated to a fibered face and gave a practically implementable algorithm for computing it. From the Teichm\"uller norm (computed from the Teichm\"uller polynomial) and the Alexander norm (computed from the Alexander polynomial) on $H^1(M,\mathbb{R})$ for a fibered hyperbolic $3$--manifold $M$, one can compute a fibered face of the unit Thurston norm ball, the cone over which contains the given fibered cohomology class. He indicates how to carry out these computations for certain pseudo-Anosov braids of the $4$--times punctured sphere.

A very detailed computation of the Alexander and Teichm\"uller norm for a particular two--cusped hyperbolic $3$--manifold, namely the sibling of the Whitehead link complement, was carried out by Aaber--Dunfield~\cite{AaberDunfield}. The authors of that paper explicitly computed the Alexander and Teichm\"uller polynomials of the sibling of the Whitehead link complement, as well as the whole Thurston norm ball, which is a square with unit side length centered at the origin. They show that all four faces of the square are fibered, and so that every primitive cohomology class of the manifold is fibered except the ones lying on the two lines passing through the corners of the square.

For other $3$--manifolds presented as mapping tori of multiply punctured disks, an algorithm for computing the Teichm\"uller polynomial
(as well as further topological data associated to other fibrations of the manifold) was proposed by Lanneau--Valdez~\cite{LanVal}.
Thus, the theory we apply in this paper is rich with computationally tractable examples.

\section{An application to public-key cryptography}\label{sec:public}

We now propose a cryptographic scheme which uses the Thurston norm. Alice and Bob are communicating over an insecure channel. The following information is public:

\begin{itemize}
\item
A fibered hyperbolic $3$--manifold $M$ with $H^1(M,\R)$ of rank at least two, a fixed finite presentation of $\pi_1(M)$, and a fibered face $F$ of the Thurston norm ball. For instance, one could use Thurston's example of the simplest pseudo-Anosov braid,
or alternatively Aaber--Dunfield's example, as described above.
\item
For every primitive integral cohomology class $\phi\in\R_+\cdot F$, a standard presentation $\mathcal{P}_{\phi}$ of the corresponding fiber group $\pi_1(S_{\phi})$, a stable letter $t_{\phi}\in\pi_1(M)$, and an automorphism $\psi_{\phi}$ of $\pi_1(S_{\phi})$
such that $\pi_1(M)$ is the semidirect product of $\pi_1(S_{\phi})$ with $\Z$ with the conjugation action of
$t_{\phi}$ given by $\psi_{\phi}$. An explicit isomorphism between $\pi_1(M)$ with its fixed presentation and this semidirect product presentation \[\langle \mathcal{P}_{\phi},t_{\phi}\mid t_{\phi}xt_{\phi}^{-1}=\psi_{\phi}(x)\rangle\]
corresponding to the fibered class $\phi$ is included in the data. We write $\Prim(\R_+\cdot F)$ for the collection of such primitive integral cohomology classes. Here, the choice of the stable letter is not strictly necessary, but it precludes the need for an extra conjugacy decision
later on in the cryptoscheme.

Strictly speaking, $\Prim(\R_+\cdot F)$ is an infinite collection of data, so it is necessary to define precisely what it means to share it. In practice,
one would have a searchable database which is large enough to give Alice and Bob a rich collection of choices, and so that over a course
of their communication, no repeats would be necessary. One could restrict
the database to primitive rational cohomology classes in $F$ whose denominators are at most some arbitrary large cutoff say $10^{12}$.
\item
A finitely generated group $G$ suitable for a public-key cryptographic scheme such as the Anshel-Anshel-Goldfeld protocol (see~\cite{AAG})
and an efficiently computable function \[f\colon G\to \Prim(\R_+\cdot F).\]
\end{itemize}

Here, a \emph{standard presentation} for a fiber group is either a presentation of a free group with finitely many generators and no relations, or the standard presentation of a closed surface group of genus $g$: \[\pi_1(S_g)=\langle a_1,b_1,\ldots,a_g,b_g\mid [a_1,b_1]\cdots [a_g,b_g]=1\rangle.\] 
The scheme is as follows:

\begin{enumerate}
\item
Alice and Bob use a public-key cryptographic scheme in order to produce a shared private key $g\in G$.
\item
Using the public function $f$, Alice and Bob agree on the fibered cohomology class $\phi=f(g)$ with corresponding fiber group $\pi_1(S)$ with presentation $\mathcal{P}$ and automorphism $\psi$.
\item
Alice chooses an arbitrary positive integer $N$ and computes the length of $\psi^N(s)$ for every generator of $\pi_1(S)$ in the presentation $\mathcal{P}$. The key is the maximum length $\ell_{\max}$ obtained in this way, on the generator $s_{\max}$ in $\mathcal{P}$.
\item
Over a public channel, Alice sends a finite collection $\{x_1,\ldots,x_t\}\subset \pi_1(M)$, written in terms of the fixed generators for $\pi_1(M)$, and of length comparable to $N$.
She chooses these elements in such a way that if $X_{\phi}$ is the generating set in the presentation $\mathcal{P}_{\phi}$, then
 $\psi^N(X_{\phi})\subset \{x_1,\ldots,x_t\}$, but so that all but $|X_{\phi}|$
 of the elements $\{x_1,\ldots,x_t\}$ do not belong to $\pi_1(S)$. We assume that $|X_{\phi}|\ll t$.
\item
Bob checks which of these elements lie in $\pi_1(S)$.
\item
Bob uses the fact that $\psi$ is given by conjugation by an element of $\pi_1(M)$ to recover $N$.
\item
Bob computes the length of $\psi^N(s)$ for each generator of $\pi_1(S)$ and recovers $\ell_{\max}$.
\end{enumerate}

We remark that the value of $\ell_{\max}$ is uniquely determined by Alice's initial choice of $N$.

\section{An explicit example}\label{sec:example}

We consider the example of a fibered $3$-manifold coming from the simplest pseudo-Anosov braid, as worked out in~\cite{McPaper}.

\subsection{The $3$-manifold}
The
initial fiber is $S_0$, which is identified with the thrice punctured disk, and whose mapping class group is identified with the three--stranded
braid group $B_3$. In the standard braid generating set $\{\sigma_1,\sigma_2\}$, the simplest pseudo-Anosov braid is given by
$\beta=\sigma_1\sigma_2^{-1}$.

We have that $\pi_1(S_0)=\langle x,y,z\rangle\cong F_3$, where these generators are identified with small based loops about the three
punctures of $S_0$. We have \[\sigma_1\colon (x,y,z)\mid (y,y^{-1}xy,z)\] and \[\sigma_2^{-1}\colon (x,y,z)\mapsto (x,yzy^{-1},y).\] Thus,
we have \[\beta\colon (x,y,z)\mapsto (yzy^{-1},yz^{-1}y^{-1}xyzy^{-1},y).\] A presentation for the fundamental group of the
fibered $3$-manifold associated to $\beta$
can be written as \[\pi_1(M)=\langle t,x,y,z\mid t^{-1}xt=yzy^{-1},t^{-1}yt=yz^{-1}y^{-1}xyzy^{-1},t^{-1}zt=y\rangle.\]

It is easy to see that $\beta$ acts transitively on the punctures of $S_0$, and so that $H_1(M,\Z)\cong\Z^2$. If $\phi\in H^1(M,\Z)$, then
$\phi$ is determined by its values on $t$ and on $x$, so that we may write $\phi=(a,b)\in\Z^2$ for the cohomology class which satisfies
$\phi(t)=a$ and $\phi(x)=b$. McMullen computes the Thurston norm on $H^1(M,\R)$ and shows that it is given by
\[||\phi||_T=\max\{|2a|,|2b|\}.\] Each face of the Thurston unit norm ball is fibered. The face $F$ whose cone contains $(1,0)$ is therefore
given by $F=\{1/2\}\times [-1/2,1/2]$.

\subsection{From a public key to a new fibration}

Let $G$ be a finitely generated group suitable for the Anshel-Anshel-Goldfeld protocol, with a fixed normal form. Then
we can construct a computable function which associates to elements of $G$ various fibrations of $M$. If $g\in G$, we write
$g$ in a normal form and let $|g|$ denote the length of $g$ in this normal form. Then, we may set
\[f(g)=D(g)\left(\{1/2\},\left\{\frac{|g|}{|g|+1}-1/2\right\}\right),\] where here $D(g)$ is chosen to be the smallest positive integer so that
$f(g)\in\Z^2$. Note that $D(g)\leq\mathrm{lcm}\{2,|g|+1\}$. Thus, we have associated to $g\in G$ a new cohomology class
 $f(g)$ which is defined
by \[f(g)(t)=D(g)/2\] and \[f(g)(x)=|g|\cdot D(g)/(|g|+1)-D(g)/2.\] This cohomology class will represent a new fibration provided $|g|\neq 0$.
We remark that the
function $f$ constructed here is just one possible example which would suit our purposes. There are many other suitable candidates
for $f$.

\subsection{New fiber subgroups}

Given $\phi\in H^1(M,\Z^2)$. We have that the new fiber subgroup $\pi_1(S_{\phi})$ is given by the kernel of $\phi$, viewed as the composition
\[\pi_1(M)\to H_1(M,\Z)\to \Z,\] where the first map is the abelianization map and where the second map is $\phi$. The group $\pi_1(S_{\phi})$
will always be a finite rank free group, and its rank can be computed as $||\phi||_T+1$, since $||\phi||_T$ denotes the
absolute value of the Euler characteristic
of the fiber. Finding a presentation for the fiber subgroup is sometimes possible~\cite{Brown,DKL}, though in general it may be difficult.
This is why we assume that free presentations for fiber subgroups are part of the public data.

\subsection{Distortion of lengths}

The advantage of the scheme we propose is that a very large integer is encoded by a relatively small one. The essential point is that
both Alice and Bob do computations in $\pi_1(M)$, essentially just conjugation. The secrecy of the scheme is entirely in the choice of
fibration, which in turn gives a mapping class and an exponentially distorted subgroup. Other than conjugation, which results in linear
growth of words in $\pi_1(M)$, Proposition~\ref{prop:linear} shows that membership in the fiber subgroup is computable in linear time.
Alice and Bob use their common knowledge of the fiber subgroup to extract an integer $N$, which is on the order of the logarithm of the
shared key. It is in this last step, passing from $N$ to $\ell_{\max}$ that the exponentially distorted subgroup comes into play.

Alice picks $N$, and her key
is the maximal length of $\psi_{\phi}^N$
applied to elements of the free generating set of the fiber subgroup. To communicate $N$ over the channel,
she applies $\psi_{\phi}^N$, viewed as conjugation by  the element of $t_{\phi}\in\pi_1(M)$. The resulting elements of $\pi_1(M)$
will have lengths which are linear in $N$.
When Alice sends information over the
channel, Bob applies $t_{\phi}$ successively to generators of $\pi_1(S_{\phi})$ and checks to see if the generating set lies in
$\{x_1,\ldots,x_t\}$. This will require linearly many calculations in $N$ and $t$, since the word problem in hyperbolic $3$-manifold
groups has linear complexity. Once Bob finds the first such $N$, this will be the same
value of $N$ as chosen by Alice, since $\psi_{\phi}$ is pseudo-Anosov and is therefore not periodic.

The key that Alice and Bob wish to share is instead $\ell_{\max}$, which is the maximal length of $\psi_{\phi}^N$ applied to a generator
of $\pi_1(S_{\phi})$, viewed as an element of $\pi_1(S_{\phi})$. Alice and Bob now both know $N$, and thus can recover $\ell_{\max}$,
since
$\psi_{\phi}$ is a public automorphism of the free group $\pi_1(S_{\phi})$. The size of $\ell_{\max}$ will be exponential in $N$.
The precise distortion can be computed from the Teichm\"uller polynomial. McMullen computes the Teichm\"uller polynomial for that fibered face
to be \[\Theta(t,u)=1-t(1+u+u^{-1})+t^2.\] Here, the polynomial
$\Theta(t,u)$ is viewed as an element of the group ring $\Z[H,t^{\pm1}]$,
where here $H=\langle u\rangle$ denotes the $\psi_{\phi}$--invariant homology of $S_{\phi}$.
In this case, the generator $t$ can be identified with the stable letter $t$
of the fibration defining $M$, and we can write the other generator as $u=[x]+[y]+[z]$, the sum of the homology classes of the three punctures.
Then, we can write \[\Theta(t,u)=\sum_{g\in H_1(M,\Z)}a_g g,\] with $a_g\in \Z$.
If $\phi\in H^1(M,\Z)$, then we view $\phi$ as an element of $\mathrm{Hom}(\pi_1(M),\Z)$ and and we can write
\[\Theta(k)=\sum_{g\in H_1(M,\Z)} a_gk^{\phi(g)}.\] We set $\lambda_{\phi}$ to be the largest root of $\Theta(k)$, which
will always be a real number which is greater than one. Then we have $\ell_{\max}\sim\lambda_{\phi}^N$.

For the specific Teichm\"uller polynomial \[\Theta(t,u)=1-t(1+u+u^{-1})+t^2,\] we work out two examples of stretch factors $\lambda_{\phi}$
for two different  fibered cohomology classes. The canonical class $\phi=(1,0)$ is the one describing the original fibered $3$--manifold.
In this  case, we compute $\phi(t)=1$ and $\phi(u)=0$. We thus obtain the polynomial $\Theta(k)=k^2-3k+1$, whose largest root  is one more than the Golden
Ratio $(3+\sqrt{5})/2$. This is well--known to be the stretch factor of $\beta$.

We also consider the class $\phi=(2,1)$. It is easy to check that this class lies in the cone over $F$. We obtain the polynomial
$\Theta(k)=k^4-k^3-k^2-k+1$, which one discovers by numerical calculation to have a root at $k\approx 1.72208$.

\subsection{Generalization to other pairs of groups}
In principle, there is nothing special about the pair $(\pi_1(M),\pi_1(S))$ for a fiber subgroup of a hyperbolic $3$--manifold. For the applications,
one could  use any pair of a finitely generated group and a finitely generated exponentially distorted subgroup (cf. \cite{CKL}). The particular
advantage of using the Thurston norm is that from a single hyperbolic group, we obtain infinitely many different exponentially distorted
subgroups which are distorted in different ways. Moreover, the Teichm\"uller polynomial organizes the different fiber subgroups in a
computationally convenient way. In general, such a framework does not exist for arbitrary such pairs of groups. A notable exception,
on which one could also base a similar cryptoscheme, is the class of free--by--cyclic groups~\cite{AHR,DKL}.

\section{An application to symmetric-key cryptography}

The scheme developed in Sections~\ref{sec:public} and~\ref{sec:example} can be simplified somewhat to yield a symmetric-key cryptographic
scheme. Such a scheme would eschew the need for an initial private shared key. The setup for this scheme would be as follows.

{\bf Public information:} The fibered hyperbolic manifold $M$, presented as a mapping torus,
would be public as before. A presentation for $\pi_1(M)$ would also be public.

{\bf Private information:} Alice and Bob would agree beforehand on a fibered cohomology class $\phi$ of $M$. This would yield a private fibered
subgroup $\pi_1(S_{\phi})$ with a preferred presentation, and a private automorphism $\psi_{\phi}$ of $\pi_1(S_{\phi})$.

The implementation of the cryptoscheme would be as follows:

\begin{enumerate}
\item
Alice chooses an arbitrary integer $N$.
\item
Alice communicates the integer $N$ to Bob over a public channel.
\item
Alice and Bob both compute $\psi_{\phi}^N$ on the generators of $\pi_1(S_{\phi})$ in the preferred presentation, and compute the lengths
of the resulting words.
\item
The shared key is $\ell_{\max}$, the longest length of a word obtained by applying $\psi_{\phi}^N$ to the generators.
\end{enumerate}

\section{Remarks on security}

There are several layers of security which are built into the schemes we propose.
In the public-key scheme, the first layer of security lies in similar security assumptions of the public key-exchange scheme \`a la
Anshel-Anshel-Goldfeld (AAG). We recall that this is used by Alice and Bob to agree on a fibered cohomology class.
We note that there have been several proposals for the platform group for the AAG protocol, such as braid groups,
polycyclic groups, Grigorchuk groups, certain classes of right-angled Artin groups and their subgroups,
in which the simultaneous conjugacy search problem is difficult.
Apart from a few proposed attacks (see~\cite{Tsaban,Romanikov} for attacks adapted to braid group and linear group platforms as well
as potential ripostes),
the AAG scheme remains secure for suitable choices of parameters.
In the symmetric-key scheme, the use of AAG is unnecessary, thus removing any vulnerabilities to attacks on that protocol.

In both schemes, the key that Alice and Bob share will be much larger than any of the public data over the channel. Thus, even if
the eavesdropper is able to guess $N$, the key $\ell_{\max}$
is exponentially longer than $N$ and would therefore be much more difficult to guess,
short of knowing which cohomology class Alice and Bob are using.

For Alice and Bob, computations inside the group $\pi_1(S)$ are easy since this group is either free or a surface group, and hence computing word lengths and solving the word problem is relatively easy (using small cancellation theory, for instance)
from the standard presentations given in the public data.

Finally, since  the public data of the fibered classes is public and  since for practical  purposes it must be truncated to be a finite list, it is
conceivable that an eavesdropper would simply use the database of fibrations and an efficient solution to the conjugacy problem to discover
the key without knowing the secret fiber. The eavesdropper could conceivably  check each fixed presentation in the database and look for
words sent by  Alice in which the stable letter appears with  exponent sum zero.
This vulnerability can be overcome by making the database very large compared to the size of the
key, so that processing the whole database would be prohibitive. Thus, the key shared by Alice and Bob would be obsolete by the time the
eavesdropper could compute it.

\section{Acknowledgements}
The authors thank the anonymous referees for their careful reading of the manuscript and for their helpful comments.

%

\begin{thebibliography}{1}

\bibitem{AaberDunfield}
J. W. Aaber and N. Dunfield, {\it Closed surface bundles of least volume},
Algebr. Geom. Topol. 10 (2010), no. 4, 2315--2342.

\bibitem{AHR}
T. Algom-Kfir, Yael, E.  Hironaka, K. Rafi,
{\it Digraphs and cycle polynomials for free-by-cyclic groups},
Geom. Topol. 19 (2015), no. 2, 1111--1154.

\bibitem{AAG} I. Anshel and M. Anshel and D. Goldfeld, {\it An algebraic method for public-key cryptography}, Math. Res. Let.6, 1999, 287--291.

\bibitem{AAGG}
I. Anshel, D. Atkins, D. Goldfeld, P. Gunnells, {\it WalnutDSATM: A Quantum Resistant Group Theoretic Digital Signature Algorithm}, \url{https://www.nist.gov/sites/default/files/documents/2016/10/19/atkins-paper-lwc2016.pdf}, 2016.

\bibitem{AAGL}
I. Anshel, M. Anshel, D. Goldfeld, and S. Lemieux, {\it Key
agreement, the Algebraic Eraser, and lightweight cryptography},
Algebraic methods in cryptography, Contemp. Math. Amer. Math. Soc.
{\bf 418} (2006),  1--34.

\bibitem{Bau}
A. Baudisch, {\it Subgroups of semifree groups}, Acta Math. Acad. Sci. Hungar. {\bf 38} (1981), no. 1-4, 19--28.

\bibitem{BFX} G. Baumslag, B. Fine, and X. Xu, {\it Cryptosystems using linear groups}, Appl. Algebra Eng. Comm. Comput., 17:205--2017, 2006.

\bibitem{BMS}
J. C. Birget, S S. Magliveras , M. Sramka, {\it On public-key cryptosystems based on combinatorial group theory}, Tatra Mt. Math. Publ., {\bf 33} (2006), 137--148.

\bibitem{MB}
M. Bridson, {\it On the subgroups of right-angled Artin groups and mapping class groups},
Math. Res. Lett., {\bf 20} (2013), 203--212.

\bibitem{Brown}
K. Brown, {\it Trees, valuations, and the Bieri-Neumann-Strebel invariant}, Inv. Math. {\bf 90} (1987), 479--504.

\bibitem{CK}   B. Cavallo, D. Kahrobaei, {\it Secret sharing using the Shortlex order and non-commutative groups},
Contemporary Mathematics {\bf 633}, AMS, 1--8, (2015).

\bibitem{charney}
R. Charney, {\it An introduction to right-angled Artin groups}, Geom. Dedicata, {\bf 125} (2007), 141--158.

\bibitem{CKL}  I. Chatterji, D. Kahrobaei, N. Y. Lu, {\it Cryptosystems Using Subgroup Distortion}, Theoretical and Applied Informatics {\bf 29}, 14--24 (2017).

\bibitem{DH}
W. Diffie  and  M.~E. Hellman,  {\it New Directions in
Cryptography}, IEEE Transactions on Information Theory {\bf IT-22}
(1976), 644--654.

\bibitem{DKL}
S. Dowdall, I. Kapovich, and C. Leininger, {\it McMullen polynomials and Lipschitz flows for free-by-cyclic groups},
J. Eur. Math. Soc. 19 (2017), no. 11, 3253--3353.

\bibitem{EK} B. Eick and D. Kahrobaei, {\it Polycyclic groups: a new platform for cryptography}, \url{math.gr/0411077} Technical report (60 citations), 2004.

\bibitem{FLP} A. Fathi, F. Laudenbach, V. Po\'enaru, {\it Travaux de Thurston sur les surfaces},
S\'eminaire Orsay. Ast\'erisque, 66--67. Soci\'et\'e Math\'ematique de France, Paris, 1979.

\bibitem{FK}
R. Flores, D. Kahrobaei, {\it Cryptography with Right-angled Artin Groups}, Theoretical and Applied Informatics, {\bf  28} (2016), no. 3, 8--16.

\bibitem{FKK} R. Flores, D. Kahrobaei, T. Koberda, {\it Algorithmic Problems in right-angled Artin groups: Complexity and Applications}, J. Algebra {\bf 519} (2019), 111--129.

\bibitem{GK} J. Gryak and D. Kahrobaei, {\it The status of the polycyclic group-based cryptography: A survey and open problems}, Groups Complexity Cryptology, 8:171--186, 2016.

\bibitem{HKKS} M. Habeeb, D. Kahrobaei, C. Koupparis, and V.
Shpilrain, {\it Public key exchange using semidirect product of
(semi)groups},  in: ACNS 2013, Applied Cryptography and Network Security, LNCS {\bf 7954}
(2013), 475--486.

\bibitem{HKS} M. Habeeb, D. Kahrobaei, and V. Shpilrain, {\it A Secret Sharing scheme based on group-presentation
and word problem}, Contemporary Mathematics, AMS {\bf 582} (2012), 143--150.

\bibitem{HKMPPQ}
D. Hart, D.H. Kim, G. Micheli, G. Pascual Perez,  C. Petit, Y. Quek, {\it A Practical Cryptanalysis of WalnutDSA}, LNCS, PKC, 2018.

\bibitem{KK6} D. Kahrobaei, B. Khan, {\it A Non-Commutative Generalization of the El Gamal Key Exchange using Polycyclic Groups}, Proceeding of IEEE, GLOBECOM (2006), 1--5.

\bibitem{KK}  D. Kahrobaei, C. Koupparis, {\it Non-commutative digital signatures using non-commutative groups}, Groups, Complexity, Cryptology {\bf 4} (2012),  377--384.

\bibitem{KM} D. Kahrobaei, K. Mallahi-Karai, {\it Some applications of arithmetic groups in cryptography}, Groups Complexity, Cryptology, De Gruyter {\bf 11}, no. 1, 1--9 (2019).


\bibitem{KS16} D. Kahrobaei, V. Shpilrain, {\it Using semidirect product
of (semi)groups in public key cryptography}, Computability in Europe 2016, LNCS {\bf 9709}, 132--141 (2016).

\bibitem{LanVal} E. Lanneau, F. Valdez, {\it Computing the Teichm\"uller polynomial},
J. Eur. Math. Soc. (JEMS) 19, no. 12, 3867--3910 (2017).

\bibitem{LP}
F. Levy-dit-Vehel, L. Perret, {\it On the Wagner-Magyarik Cryptosystem}. In: Ytrehus (eds) Coding and Cryptography. LNCS, {\bf 3969}, Springer, Berlin, Heidelberg, 2006.


\bibitem{LWZ}
H. Liu, C. Wrathall, K. Zeger, {\it Efficient solution of some problems in free partially commutative
monoids}, Information and Computation, {\bf 89} (1990), 180--198.

\bibitem{McPaper}
C. T. McMullen, {\it Polynomial invariants for fibered 3-manifolds and Teichm\"uller geodesics for foliations}, Ann. Sci. \'Ecole Norm. Sup. (4) 33 (2000), no. 4, 519--560.

\bibitem{P} G. Petrides, {\it Cryptanalysis of the public key cryptosystem based on the word problem on the Grigorchuk groups}, Cryptography and coding, 234--244, 2003.

\bibitem{Romanikov}
V. Roman'kov, {\it An improved version of the AAG cryptographic protocol}, Groups Complex. Cryptol. (2019), appear.

\bibitem{SU} V. Shpilrain and A. Ushakov {\it Thompson's group and public key cryptography}, ACNS, 2005.

\bibitem{SZ} V. Shpilrain and G. Zapata, {\it Combinatorial group theory and public key cryptography}, Applicable Algebra in Engineering, 17, 291--302, 2006.

\bibitem{Thurston}
W. P. Thurston, {\it A norm for the homology of 3-manifolds},
Mem. Amer. Math. Soc. 59 (1986), no. 339, i--vi and 99--130.

\bibitem{Tsaban}
B. Tsaban, {\it Polynomial-time solutions of computational problems in noncommutative-algebraic cryptography},
J. Cryptology {\bf 28} (2015), no. 3, 601--622.

 \bibitem{WM}
N. R. Wagner, M. R. Magyarik, {\it A Public-Key Cryptosystem Based on the Word Problem} In: Blakley G.R., Chaum D. (eds) Advances in Cryptology. CRYPTO 1984. Lecture Notes in Computer Science, {\bf 196}, Springer, Berlin, Heidelberg, 1985.



\end{thebibliography}

\end{document}